\documentclass[12pt]{amsart}
\usepackage{amsmath,amssymb,amsfonts,amsthm,amsopn,combelow}
\usepackage{graphics}

\newcommand{\tfa}{time-frequency analysis}

\newcommand{\ft}{Fourier transform}

\newcommand{\tf}{time-frequency}

\newcommand{\modsp}{modulation space}
\newcommand{\psdo}{pseudodifferential operator}

\newtheorem{tm}{Theorem}[section]
\newtheorem{lemma}[tm]{Lemma}

\newtheorem{theorem}{Theorem}[section]

\newtheorem{proposition}[theorem]{Proposition}%%
\newtheorem{remark}[theorem]{Remark}

\newcommand{\beqa}{\begin{eqnarray*}}
\newcommand{\eeqa}{\end{eqnarray*}}

\newcommand{\field}[1]{\mathbb{#1}}
\newcommand{\bR}{\field{R}}        %  real numbers
\def\la{\lambda}

 \def\cF{\mathcal{F}}              % Calligraphic Letters
 \def\cS{\mathcal{S}}

\def\rd{\bR^d}

\def\rdd{{\bR^{2d}}}

\def\intrdd{\int_{\rdd}}

\def\R{\right)}

\def\<{\left<}
\def\>{\right>}

\def\mv1{M_v^1}

\def\Mmpq{M_m^{p,q}}
\def\phas{(x,\xi )}

\def\o{\xi}

\def\z{\zeta}

\def\R{\mathbb{R}}
\def\Ren{\mathbb{R}^d}
\def\Renn{\mathbb{R}^{2d}}

\def\sch{\mathcal{S}}

\def\Fur{\mathcal{F}}

\def\f{\varphi}

\def\Opw{Op_{\mathrm{W}}}
\def\Sn2{S_{2}(L^{2}(\Ren))}
\def\S1{S_{1}(L^{2}(\Ren))}
\def\sig00{\sigma_{0,0}}
\def\la{\langle}
\def\ra{\rangle}

\begin{document}

%----------Author 1

\title[Boundedness of Pseudodifferential Operators]{Boundedness of Pseudodifferential Operators with symbols in Wiener amalgam spaces on Modulation Spaces}
\author{Lorenza D'Elia}
\address{Dipartimento di Matematica,
Universit\`a di Torino, via Carlo Alberto 10, 10123 Torino, Italy}
\author{S. Ivan Trapasso}
\address{Dipartimento di Matematica,
Universit\`a di Torino, via Carlo Alberto 10, 10123 Torino, Italy}
\email{lorenza.delia@edu.unito.it}
%\email{maurice.de.gosson@univie.ac.at}
\email{salvatore.trapasso@edu.unito.it}
%\thanks{This work was completed with the support of our
%----------Author 2

%\thanks{{\bf Acknowledgments.} The authors thank Luigi Rodino  for very helpful  discussions on this
%topic.}
\date{}
%\dedicatory{To my parents}
%%% ----------------------------------------------------------------------

\begin{abstract}
This paper provides sufficient conditions for the boundedness of Weyl operators on modulation spaces. The Weyl symbols belong to Wiener amalgam spaces, or generalized modulation spaces, as recently renamed by their inventor Hans Feichtinger. This is the first result which relates symbols in Wiener amalgam spaces to operators acting on classical modulation spaces.
\end{abstract}

%%% ----------------------------------------------------------------------
\maketitle
%%% ----------------------------------------------------------------------

\section{Introduction}
In this paper we investigate the boundedness properties of pseudodifferential operators in the Weyl form. These operators  arise as quantization rule proposed by Weyl in \cite{Weyl1927}. Namely, the rule  assigns an operator $\Opw(a)$ to a function $a$ (the so-called Weyl symbol) on the phase space $\rdd$:
$$ a\to \Opw(a).$$
The operator $ \Opw(a)$ is called a Weyl operator or Weyl transform (cf., e.g., \cite{Wongbook}).
From a Time-frequency Analysis perspective Weyl operators can be introduced by means of the related time-frequency representation, the so-called (cross-)Wigner distribution $W(f,g)$, which for signals 
$f,g$ in the
Schwartz class $\mathcal{S}(\mathbb{R}^{d})$ is defined by 
\begin{equation}
W(f,g)(x,\omega )=\int_{\mathbb{R}^{d}}e^{-2\pi iy\omega }f(x+\frac{y}{2})\overline{g(x-\frac{y}{2})}\,dy.  \label{wigner}
\end{equation}
The Weyl operator ${Op}_{\mathrm{W}}(a)$ with symbol $a$ in the space of tempered distribution $\mathcal{S}^{\prime }({\mathbb{R}^{2d}})$ can be then defined by the formula 
\begin{equation}\label{weakdef}
\langle {Op}_{\mathrm{W}}(a)f,g\rangle =\langle
a,W(g,f)\rangle,\quad f,g\in \cS(\rd).
\end{equation}

The study of continuity properties for Weyl operators on different kinds of function spaces has been pursued by many authors.  Depending on the  properties of the symbol $a$, one can infer the corresponding continuity properties of the related operator ${Op}_{\mathrm{W}}(a)$. 

For the continuity properties of ${Op}_{\mathrm{W}}(a)$ on $L^p(\rd)$ spaces we refer the reader to \cite{cordero2,Wongbook}.\par 
Here we focus on Banach spaces which measure
the time-frequency decay of a function/distribution in the phase space.  They are called modulation and Wiener amalgam spaces. Indeed, we shall study the continuity properties of the operator ${Op}_{\mathrm{W}}(a)$ on the modulation spaces $M^{r_1,r_2}(\rd)$ $1\leq r_1,r_2\leq\infty$ (cf. the following section for their definition), introduced by Hans Feichtinger in \cite{feichtinger83}. The corresponding Weyl symbol $a$  belongs to the Wiener amalgam spaces $W(\cF L^p, L^q)$, $1\leq p,q\leq \infty$ (cf. Section 2). 
The latter spaces are often known in the literature as Wiener amalgam spaces with local component $\cF L^p$ and global component $L^q$, for $1\leq p,q\leq \infty$,  but nowadays their inventor Hans Feichtinger \cite{feichtinger90} is  suggesting to call them simply modulation spaces,
since they arise as the Fourier transform of the classical modulation spaces $M^{p,q}$ introduced in \cite{feichtinger83} and can similarly be defined  by means of the short-time Fourier transform (see Section 2 for details).

Continuity properties of Weyl operators with symbols in classical modulation spaces $M^{p,q}$ have been investigated by many authors, starting from the earliest paper \cite{g-heil99}. The most important contributions in this framework are contained in \cite{beltita,benyi,CG02,cgn01,cgn3,cordero1,cordero2,corderonicolasharp,cordero3,corderoToftWah,charly06,labate2,sugitomita2,toft1,Toftweight}.\par 
Let us also recall the many studies on the continuity properties of Fourier integral operators (FIOs) on modulation spaces \cite{fiowiener,cn,fio5,fio9,fio1,fio3,fio-evolution,fio6,fio7,fio8,fio09,fio10} which  find applications principally in  the study of Schr\"odinger equations. Pseudo-differential operators are a special case of FIOs, having phase function $\Phi\phas=2\pi i x\xi$.  

This study is limited to pseudodifferential operators, however a future object of our research would be to investigate the continuity properties for FIOs.

The main result of this paper can be formulated in the un-weighted case as follows (cf. the subsequent Theorem \ref{contwiener}).

\begin{theorem}\label{contwiener0} Assume that $1 \leq p,q,r_1,r_2 \leq \infty$ satisfy
\begin{equation*}
q\leq p'
\end{equation*}
and
\begin{equation*}
\max \{r_1,r_2,r_1',r_2'\}\leq p.\end{equation*}
Then every Weyl operator $\operatorname*{Op}\nolimits_{\mathrm{W}}(a)$ having symbol $a\in W(\cF L^p, L^q)$, from $\cS(\rd)$ to $\cS'(\rd)$, extends uniquely
to a bounded operator on $\mathcal{M}^{r_1,r_2}(\R^d)$, with the estimate
\begin{equation*}
\|\operatorname*{Op}\nolimits_{\mathrm{W}}(a)f\|_{\mathcal{M}^{r_1,r_2}} \lesssim
\|a\|_{W(\cF L^p,L^q)}\|f\|_{\mathcal{M}^{r_1,r_2}}.
\end{equation*}
\end{theorem}
To our knowledge, this is the first result in the literature which links symbols in Wiener amalgam spaces to operators acting on modulation spaces.

Boundedness results for Weyl operators with symbols in modulation spaces still hold for the other forms of pseudodifferential operators, the so-called $\tau$-operators. These operators can be either defined as a quatization rule or by means of the related time-frequency representation (cf. \cite{bogetal}). Here we simply recall the latter.
For $\tau \in \lbrack 0,1]$, the
(cross-)$\tau $-Wigner distributions is given by
\begin{equation}
W_{\tau }(f,g)(x,\omega )=\int_{\mathbb{R}^{d}}e^{-2\pi iy\zeta }f(x+\tau y)%
\overline{g(x-(1-\tau )y)}\,dy\quad f,g\in \mathcal{S}(\mathbb{R}^{d}),
\label{tauwig}
\end{equation}
whereas the $\tau$-pseudodifferential
 operators is 
 \begin{equation}
 \langle {Op}_{\mathrm{\tau }}(a)f,g\rangle =
 \langle
 a,W_{\tau }(g,f)\rangle \quad f,g\in \mathcal{S}(\mathbb{R}^{d}).
 \label{tauweak}
 \end{equation}
For $\tau=1/2$  we recapture the Weyl operator, if  $\tau=0$ the operator is called the Kohn-Nirenberg operator ${Op}_{\mathrm{KN}}$.
A Kohn-Nirenberg operator ${Op}_{\mathrm{KN}}$ and a Weyl operator ${Op}_{\mathrm{W}}$ are related by the formula
 	$$ {Op}_{\mathrm{KN}}(a)={Op}_{\mathrm{W}}(\mathcal{U}^{-1}a)
 	$$
 	where
 	\begin{equation}
 	\label{U}
 	\mathcal{U}^{-1}=\cF^{-1} \mathcal{N_C}\cF ,
 	\end{equation}
 	 $\cF$ is the Fourier transform, $\mathcal{N_C} f(z)= e^{-\pi i z \cdot C z}f(z)$, $z\in\rdd$, and $$C=\begin{pmatrix}
 	0 & 1/2 I \\
 	1/2 I & 0\end{pmatrix}.$$
 	An easy computation (cf. \cite[Corollary 14.5.5]{book}) shows that
 	$$| V_\Phi (\mathcal{U}^{-1}a)(z,\zeta)|=| V_{\mathcal{U}\Phi}a(z-C\zeta,\zeta)|
 	$$
 	from which we conclude that  $M^{p,q}$ is invariant under the action of $\mathcal{U}^{-1}$ and therefore, results for Kohn-Nirenberg \psdo s  with symbols in $M^{p,q}$ still hold for Weyl operators and viceversa.\par
 More generally,  for $\tau$-\psdo \,
 	it was proved in \cite{Hormander3} and in \cite[Remark 1.5]{toft1} that for every choice $\tau_1,\tau_2\in [0,1]$, $a_1,a_2\in \cS'(\rdd)$,
 	\begin{equation}\label{linktausymb1}
 	{Op}_{\mathrm{\tau_1}}(a_1)= {Op}_{\mathrm{\tau_2}}(a_2)\, \Leftrightarrow\,\widehat{a_2}(\xi_1,\xi_2)=e^{-2\pi i(\tau_2-\tau_1)\xi_1\xi_2}\widehat{a_1}(\xi_1,\xi_2).
 	\end{equation}
 	For $t>0$ consider $H_t(x,\xi)=e^{2\pi i t x \xi}$ and observe that
 	\begin{equation}\label{ftchirp}\cF H_t (\zeta_1,\zeta_2)=\frac 1{t^{d}} e^{-2\pi i\frac1t \zeta_1\zeta_2 }.\end{equation}
 	So, for $\tau_1\not=\tau_2$, by \eqref{ftchirp},
 	\begin{equation}\label{linktausymb2}
 	a_2 \phas =\frac1{|\tau_1-\tau_2|^d} e^{2\pi i(\tau_2-\tau_1)\Phi}\ast a_1 \phas,
 	\end{equation}
 	where $\Phi\phas =x\o$.
 	The mapping $a\mapsto T_\Phi a= e^{2\pi i\Phi}\ast a$ is a homeomorphism on $M^{p,q}(\rdd)$, $1\leq p,q\leq \infty$, \cite[Proposition 1.2 (5)]{toft1}.
 	
 	Coming back to  Wiener amalgam spaces $W(\cF L^p,L^q)$, we first observe that they are not invariant under the action of the operator $ \mathcal{U}=\cF^{-1} \mathcal{N_{-C}}\cF $. This is proved in \cite[Proposition 6.4]{cordero2}. So that
 	boundedness results for Weyl operators do not extend automatically to Kohn-Niremberg ones and vice-versa. This result easily extends to the case of any $\tau$-\psdo. Indeed, for any $\tau>0$, the same arguments as in the proof of Proposition 6.4 of \cite{cordero2} apply to the metaplectic operator $\mathcal{U_\tau}:=\cF^{-1} \mathcal{N_{-\tau C}}\cF$.
 	This is the reason why our main result can be stated   only for Weyl operators.
 	
We shall pursue the study of boundedness properties of  $\tau$-\psdo s in a subsequent paper.

\textbf{Notation.} We define $t^2=t\cdot t$, for $t\in\rd$, and
$xy=x\cdot y$ is the scalar product on $\Ren$. The Schwartz class is denoted by  $\sch(\Ren)$, the space of tempered
distributions by  $\sch'(\Ren)$.   We use the brackets  $\la
f,g\ra$ to denote the extension to $\sch (\Ren)\times\sch '(\Ren)$ of
the inner product $\la f,g\ra=\int f(t){\overline {g(t)}}dt$ on
$L^2(\Ren)$. 
The Fourier transform of a function $f$ on $\rd$ is normalized as
\[
\Fur f(\xi)= \int_{\rd} e^{-2\pi i x\xi} f(x)\, dx.
\]
\section{Preliminaries}

\subsection{Modulation and Wiener amalgam spaces}
Modulation and Wiener amalgam space norms are a  measure
of
the joint time-frequency distribution of $f\in \sch '$. For their
basic properties we refer to  \cite{feichtinger80,feichtinger83,feichtinger90} and the textbooks \cite{Birkbis,book}.

Let $f\in\cS'(\rd)$. We define the short-time Fourier transform of $f$ as
\begin{equation}\label{STFTdef}
V_gf(z)=\Fur [fT_x g](\xi)=\int_{\Ren}
 f(y)\, {\overline {g(y-x)}} \, e^{-2\pi iy \o }\,dy
\end{equation}
for $z=(x,\xi)\in\rd\times\rd$.\par

For  description of  decay  properties, we use
 weight functions  on the \tf\ plane. In the sequel $v$ will always be a
continuous, positive,  even, submultiplicative  weight function (i.e. a
submultiplicative weight), i.e., $v(0)=1$, $v(z) = v(-z)$, and
$ v(z_1+z_2)\leq v(z_1)v(z_2)$, for all $z, z_1,z_2\in\Renn.$
A positive, even weight function $m$ on $\Renn$ is called  {\it
  v-moderate} if
$ m(z_1+z_2)\leq Cv(z_1)m(z_2)$  for all $z_1,z_2\in\Renn.$ Let us denote by $\mathcal{M}_{v}(\Renn)$ the space of $v$-moderate weights.

Given $g\in\sch(\Ren)\setminus\{0\}$, a $v$-moderate weight
function $m$ on $\Renn$, $1\leq p,q\leq
\infty$, the {\it
  modulation space} $M^{p,q}_m(\Ren)$ consists of all tempered
distributions $f\in\sch'(\Ren)$ such that $V_gf\in L^{p,q}_m(\Renn )$
(weighted mixed-norm spaces). The norm on $M^{p,q}_m$ is
$$
\|f\|_{M^{p,q}_m}=\|V_gf\|_{L^{p,q}_m}=\left(\int_{\Ren}
  \left(\int_{\Ren}|V_gf(x,\o)|^pm(x,\o)^p\,
    dx\right)^{q/p}d\o\right)^{1/q}  \,
$$
(obvious modifications for $p=\infty$ or $q=\infty$). If $p=q$, we write $M^p_m$ instead of $M^{p,p}_m$, and if $m(z)\equiv 1$ on $\Renn$, then we write $M^{p,q}$ and $M^p$ for $M^{p,q}_m$ and $M^{p,p}_m$.

The space  $\Mmpq (\Ren )$ is a Banach space
whose definition is independent of the choice of the window $g$, in the sense that different  non-zero window functions yield equivalent  norms.
The modulation space $M^{\infty,1}$ is also called Sj\"ostrand's class \cite{Sjostrand1}.

For any $p,q \in [1,\infty]$ and any $m \in \mathcal{M}_{v}(\rdd)$,
the inner product $\la \cdot,\cdot \ra$ on $\cS (\rd) \times \cS (\rd)$ extends to a continuous sesquilinear map
$M^{p,q}_m (\rd) \times M^{p',q'}_{1/m} (\rd) \rightarrow \mathbb C$.

Here and elsewhere the conjugate exponent $p'$ of $p \in [1,\infty]$ is defined by $1/p+1/p'=1$.
For any \emph{even} weight functions $u,w$ on $\rd$, the Wiener amalgam spaces $W(\Fur L^p_u,L^q_w)(\rd)$ are given by the distributions $f\in\cS'(\rd)$ such that
\[
\|f\|_{W(\Fur L^p_u,L^q_w)(\rd)}:=\left(\int_{\Ren}
  \left(\int_{\Ren}|V_gf(x,\o)|^p u^p(\o)\,
    d\o\right)^{q/p} w^q(x)d x\right)^{1/q}<\infty  \,
\]
(obvious modifications for $p=\infty$ or $q=\infty$).
Using Parseval identity in \eqref{STFTdef}, we can write the so-called fundamental identity of \tfa\, $V_g f(x,\o)= e^{-2\pi i x\o}V_{\hat g} \hat f(\o,-x)$, so that $$|V_g f(x,\o)|=|V_{\hat g} \hat f(\o,-x)| = |\mathcal F (\hat f \, T_\o \overline{\hat g}) (-x)|$$  and (recall $u(x)=u(-x)$)
 $$
\| f \|_{{M}^{p,q}_{u\otimes w}} = \left( \int_{\rd} \| \hat f \ T_{\o} \overline{\hat g} \|_{\cF L^p_u}^q w^q(\o) \ d \o \right)^{1/q}
= \| \hat f \|_{W(\cF L_u^p,L_w^q)}.
$$
Hence Wiener amalgam spaces are simply the image under \ft\, of modulation spaces:
\begin{equation}\label{W-M}
\cF ({M}^{p,q}_{u\otimes w})=W(\cF L_u^p,L_w^q).
\end{equation}
For completeness, let us recall the inclusion properties of modulation spaces.
Suppose $m_1, m_2 \in \mathcal{M}_{v}(\rdd)$. Then
\begin{equation}\label{modspaceincl1}
\begin{aligned}
& \cS (\rd) \subseteq M_{m_1}^{p_1,q_1} (\rd) \subseteq M_{m_2}^{p_2,q_2} (\rd) \subseteq
\sch'(\rd), \\
& p_1 \leq p_2, \quad q_1 \leq q_2, \quad m_2 \lesssim
m_1.
\end{aligned}
\end{equation}

We denote by $J$ the symplectic  matrix
\begin{equation}\label{J}J=%
\begin{pmatrix}
0_{d\times d} & I_{d\times d}\\
-I_{d\times d} & 0_{d\times d}%
\end{pmatrix}.
\end{equation}

%The symplectic Fourier transform of a function $F$ on the phase space $\rdd$ is
%\[
%\Fur_\sigma F(\zeta)=\int_{\rdd} e^{-2\pi {i}\sigma(\zeta,z)} F(z)\, dz.
%\]
%We observe that the symplectic Fourier transform is an involution, i.e. $\Fur_\sigma(\Fur_\sigma F)=F$, and moreover $\Fur_\sigma F(\zeta)= \Fur F(J \zeta)$.
\par
\section{Symbols in Wiener amalgam spaces}
We need first to  investigate  the properties of the Wigner distribution in terms of Wiener amalgam spaces.
From now on we set $v_J(z)=v(Jz)$, where $J$ is the symplectic matrix in \eqref{J}. We obtain the following results.

\begin{lemma}\label{Wigwienerlemma} Consider $m \in \mathcal{M}_{v}(\rdd)$, $1\leq p_1,p_2\leq \infty$,  $f\in M^{p_1,p_2}_m$, $g\in M^{p_1',p'_2}_{1/m}$, then
the Wigner distribution $W(g,f)\in W(\cF L^1_{1/v_J}, L^\infty)$, with
\begin{equation}\label{Wigwiener}
\|W(g,f)\|_{W(\cF L^1_{1/v_J}, L^\infty)} \lesssim \|f\|_{M^{p_1,p_2}_m}\|g\|_{M^{p_1',p'_2}_{1/m}}.
\end{equation}
\end{lemma}
\begin{proof}
  If $\zeta = (\z_1,\z_2)\in \Renn$, then
\cite[Lemma~14.5.1]{book} says that
$$|{ {V}}_\Phi (W(g,f))(z,\zeta)| =| V_\f f(z
+\tfrac{J{\z }}{2})| \,  |V_\f g(z - \tfrac{J{\z }}{2})| \,.$$
Consequently
\begin{equation}\label{star}
  \|W(g,f)\|_{W(\cF L^1_{1/v_J}, L^\infty)} \asymp  \sup_{z\in\rdd} \intrdd | V_\f f(z +\tfrac{J{\z }}{2})| \,  |V_\f g(z -\tfrac{J{\z }}{2})| \frac1 {v(J\z)}  \, d\z.
\end{equation}
  Making the change of variables $u=J\z$ and observing that $$\frac 1{v(u)}\leq C \frac{m(z+\frac u 2)}{m(z-\frac u2)},$$
  \begin{align*}
    \|W(g,f)\|_{W(\cF L^1_{1/v_J}, L^\infty)}&\leq C \sup_{z\in\rdd} \intrdd  | V_\f f(z +\frac u{2})| \,  |V_\f g(z -\frac{u}{2})| \frac {m(z+\frac u2)}  {m(z-\frac u2)}  \, du\\
    &=2^{2d} C \sup_{z\in\rdd} \intrdd | V_\f f(z +u)| \,  |V_\f g(z -u)| \frac{m(z+u)} {m(z-u)}  \, du\\
    &\leq \tilde{C} \|V_{\f} f m\|_{L^{p_1,p_2}} \|V_{\f} g \frac1m\|_{L^{p'_1,p'_2}}\\
    &\lesssim \|f\|_{M^{p_1,p_2}_m}\|g\|_{M^{p_1',p'_2}_{1/m}}.
  \end{align*}
  The claim is proved.
\end{proof}
\begin{lemma}\label{Wigwienerlemma2} Consider $m \in \mathcal{M}_{v}(\rdd)$,  $f\in M^{2}_m$, $g\in M^{2}_{1/m}$, then
the Wigner distribution $W(g,f)\in W(\cF L^2_{1/v_J}, L^2)$, with
\begin{equation}\label{Wigwiener2}
\|W(g,f)\|_{W(\cF L^2_{1/v_J}, L^2)} \lesssim \|f\|_{M^{2}_m}\|g\|_{M^{2}_{1/m}}.
\end{equation}
\end{lemma}
\begin{proof}
The technique is similar to the one in  Lemma \ref{Wigwienerlemma}.  Using \eqref{star} and the change of variables $w=z +{J\z}/ {2}$, $u=J\z$, we can write
\begin{align*}
   \|W(g,f)\|_{W(\cF L^2_{1/v_J}, L^2)} &\asymp  \left( \intrdd  \intrdd  | V_\f f(z +\frac {J\z} {2})|^2 \,  |V_\f g(z -\frac{J\z}{2})|^2 \frac 1{v^2(J\z)}  d\z  \, dz\right)^\frac 12\\
    &=\left(\intrdd\intrdd | V_\f f(w)|^2 \,  |V_\f g(w-u)|^2 \frac{1}{v^2(u)}   \, du dw\right)^\frac12\\
    &\leq \tilde{C} \left(\intrdd (|V_{\f} f|^2 m^2)\ast (|V_{\f} g|^2 \frac1{m^2})\, du\right)^\frac12\\
    & \lesssim \| |V_{\f} f|^2 m^2\|_1 \||V_{\f} g|^2 \frac1{m^2}\|_1\\
    &\lesssim \|f\|_{M^{2}_m}\|g\|_{M^{2}_{1/m}},
  \end{align*}
  where we have used Young's Inequality $L^1\ast L^1\subset L^1$. This concludes the proof.
\end{proof}

\subsection{Main result}
We address this section to the study of \psdo s  acting on \modsp s and having symbols in weighted Wiener amalgam spaces.

Here is our main result.
\begin{theorem}\label{contwiener} Assume that $1 \leq p,q,r_1,r_2 \leq \infty$ satisfy
\begin{equation}\label{e1}
q\leq p'
\end{equation}
and
\begin{equation}\label{e2}
\max \{r_1,r_2,r_1',r_2'\}\leq p.\end{equation}
Consider $m \in \mathcal{M}_{v}(\rdd)$. Then every Weyl operator $\operatorname*{Op}\nolimits_{\mathrm{W}}(a)$ having symbol $a\in W(\cF L^p_{v_J}, L^q)$, from $\cS(\rd)$ to $\cS'(\rd)$, extends uniquely
to a bounded operator on $\mathcal{M}^{r_1,r_2}_m(\R^d)$, with the estimate
\begin{equation}\label{stimae1}
\|\operatorname*{Op}\nolimits_{\mathrm{W}}(a)f\|_{\mathcal{M}_m^{r_1,r_2}} \lesssim
\|a\|_{W(\cF L^p_{v_J},L^q)}\|f\|_{\mathcal{M}_m^{r_1,r_2}}.
\end{equation}
\end{theorem}
The proof uses complex interpolation between Wiener amalgam spaces $W(\cF L^\infty_{v_J},L^1)$ and  $W(\cF L^2_{v_J},L^2)$, for which we first show the corresponding boundedness results.

\begin{proposition}\label{Prop1} Consider $m \in \mathcal{M}_{v}(\rdd)$ and $a \in W(\cF L^\infty_{v_J},L^1)$. Then the operator $\operatorname*{Op}\nolimits_{\mathrm{W}}(a)$ is bounded on $M^{r_1,r_2}_m$, for every $1\leq r_1,r_2\leq\infty$, with
\begin{equation}\label{stimae0}
\|\operatorname*{Op}\nolimits_{\mathrm{W}}(a)f\|_{{M}_m^{r_1,r_2}} \lesssim
\|a\|_{W(\cF L^\infty_{v_J},L^1)}\|f\|_{{M}_m^{r_1,r_2}}.
\end{equation}
\end{proposition}
\begin{proof} For every $f\in {M}_m^{r_1,r_2}$ and $g\in {M}_{1/m}^{r'_1,r'_2}$, we can write, for any fixed $\Phi\in\cS(\rdd)\setminus \{0\}$,
\begin{equation*}
| \la \operatorname*{Op}\nolimits_{\mathrm{W}}(a) f, g\ra |=|\la a, W(g,f)\ra|\leq \|V_\Phi a\|_{L^1_z(L^\infty_{{v_J},\z})}\|V_{\Phi} W(g,f)\|_{L^\infty_z(L^1_{{{1/v_J}},\z})}.
\end{equation*}
Observe that
$$ \|W(g,f)\|_{W(\cF L^1_{v_J},L^\infty)}\asymp \|V_\Phi W(g,f)\|_{L^\infty_z(L^1_{{1/v_J},\z})}\lesssim \|f\|_{{M}_m^{r_1,r_2}} \|g\|_{{M}_{1/m}^{r'_1,r'_2}},
$$
by Lemma \ref{Wigwienerlemma}. This concludes the proof.
\end{proof}

\begin{proposition}\label{Prop2} Consider $m \in \mathcal{M}_{v}(\rdd)$ and $a \in W(\cF L^2_{v_J},L^2)$. Then the operator $\operatorname*{Op}\nolimits_{\mathrm{W}}(a)$ is bounded on $M^{2}_m$ with
\begin{equation}\label{stimae02}
\|\operatorname*{Op}\nolimits_{\mathrm{W}}(a)f\|_{{M}_m^{2}} \lesssim
\|a\|_{W(\cF L^2_{v_J},L^2)}\|f\|_{{M}_m^{2}}.
\end{equation}
\end{proposition}
\begin{proof}
The arguments are the same as Proposition \ref{Prop1}, with Lemma \ref{Wigwienerlemma} replaced by \ref{Wigwienerlemma2}. We leave the details to the interested reader.
\end{proof}
\begin{remark} (i) Observe that by \eqref{W-M}, $W (\cF L^2_{v_J},L^2)=\cF M^2_{v_J\otimes 1}$ and a straightforward modification of \cite[Theorem 11.3.5 (c)]{book}
gives
$$ \cF M^2_{v_J\otimes 1}= M^2_{1\otimes v_{J^{-1}}}= M^2_{1\otimes v_{J}}
$$
since by assumption $v(-z)=v(z)$.\par
(ii)  Since $v(-z)=v(z)$, the weight $v_J$ is even and the conclusion of the previous step (i)  also follows by \cite[Theorem 6]{feichtinger90}, in the case $p=2$.\par
(iii) Using (i) or (ii) we derive that the Wiener amalgam space $W(\cF L^2_{v_J},L^2)$ coincides with the modulation space  $M^2_{1\otimes v_J}$. Then the conclusion of Proposition \ref{Prop2} also follows from \cite[Theorem 4.3]{Toftweight}.
\end{remark}

\begin{proof}[Proof of Theorem \ref{contwiener}.] We make use of complex interpolation between Wiener amalgam and modulation spaces, using the boundedness results of Propositions \ref{Prop1} and \ref{Prop2}.
For $\theta\in [0,1]$, we have
$$ [W(\cF^{\infty}_{v_J}, L^1), W(\cF^{2}_{v_J}, L^2)]_{\theta}=W(\cF L^{p}_{v_J}, L^{p'}),
$$
with $2\leq p\leq\infty$. As far as modulation spaces concern,
$ [M^{s_1,s_2}_m M^2_m]_\theta=M^{r_1,r_2}_m$,
with
$$\frac1{r_1}=\frac {1-\theta}{s_1}+\frac\theta 2=\frac {1-\theta}{s_1}+\frac1p
$$
 and
$$\frac1{r_2}=\frac {1-\theta}{s_2}+\frac\theta 2=\frac {1-\theta}{s_2}+\frac1p
$$
hence $r_1,r_2\leq p$. Similarly we obtain $r_1',r_2'\leq p$, and the \eqref{e2} follows.
Finally, inclusion relations for Wiener amalgam spaces allow to consider symbols $a\in W(\cF L^{p}_{v_J}, L^{q})$, with $q\leq p'$, which gives \eqref{e1} and concludes the proof.
\end{proof}

\section*{Acknowledgements}
The authors would like to thank Professors Elena Cordero and Fabio Nicola for fruitful conversations and comments.

% \begin{proof} We have, by  \cite[Theorem 6.10]{bogetal},  that
% $$\la \operatorname*{Op}\nolimits_{\mathrm{BJ}}(a)f,g \ra= \la a, Q(g,f)\ra=\la a,W(g,f)\ast\Theta_\sigma \ra \quad f,g\in \cS(\rd).
% $$
% Now, for every $f,g\in\cS(\rd)$, we have $W(g,f)\in\cS(\rdd)$ and by the observation above $W(g,f)\ast\Theta_\sigma\in \cS(\rdd)$. Hence
%\begin{align*}|\la \operatorname*{Op}\nolimits_{\mathrm{BJ}}(a)f,g \ra|&\leq \|a\|_{1}\| Q(g,f)\|_{\infty}\\
%& =\|a\|_{1}\| W(g,f)\ast\Theta_\sigma\|_{\infty}\\
%&\leq \|a\|_1 \| W(g,f)\ast\Theta_\sigma\|_{M^{\infty,1}}.\\
%\end{align*}
%Indeed, $M^{\infty,1} \hookrightarrow \cC_b \hookrightarrow L^\infty$, see, for instance, Boulkhemair \cite[Theorem 1.2]{boulkhemair}.
%Convolution relations for \modsp s \eqref{mconvm} give
%$$ M^{\infty,1}\ast M^{1,\infty}\hookrightarrow M^{\infty,1}
%$$
%so that we can continue the previous majorization as
%$$|\la \operatorname*{Op}\nolimits_{\mathrm{BJ}}(a)f,g \ra|\leq \|a\|_1 \| W(g,f)\|_{M^{\infty,1}}  \|\Theta_\sigma\|_{M^{1,\infty}}.\\
%$$
%Indeed, by Proposition \ref{sincxpwiener} $\Theta\in W(\cF L^1,L^\infty)$, so that $\Theta_\sigma \in M^{1,\infty}$.
%Finally,
%$$ \| W(g,f)\|_{M^{\infty,1}}
%$$
%where the last inequality is due to Lieb \cite[Theorem 1]{Lieb}, see also \cite[Proposition 6.3]{bogetal}.
%
%Taking the supremum over  the sets $\{ \|g\|_{p'},\,g\in\cS(\rd)\} $ and $\{\|f\|_{p}\,f\in\cS(\rd)\}$ we obtain
%$\|\operatorname*{Op}\nolimits_{\mathrm{BJ}}(a)\|_{B(L^p)}\leq C \|a\|_{1}$,
%as desired.
%\end{proof}

\vskip0.5truecm
%\textit{Acknowledgments}. The authors thank Luigi Rodino for very helpful  discussions on this topic.

\end{document}